\documentclass[14pt]{amsart}
\usepackage{latexsym, amssymb}
\usepackage[usenames]{color}
\usepackage[all]{xy}
\usepackage{graphicx,color,tikz}
\usepackage{hyperref}

\newcommand{\be}{\begin{equation}}
\newcommand{\ee}{\end{equation}}
\newcommand{\beq}{\begin{eqnarray}}
\newcommand{\eeq}{\end{eqnarray}}

\newtheorem{thm}{Theorem} [section]
\newtheorem{cor}[thm]{Corollary}
\newtheorem{lem}[thm]{Lemma}

\newtheorem{prop}[thm]{Proposition}
\theoremstyle{definition}
\newtheorem{defn}[thm]{Definition}
\theoremstyle{remark}

\newtheorem{con}[thm]{Conjecture}
\numberwithin{equation}{section}

\begin{document}
\title[Generic Newton polygons for $L$-functions of $(A,B)$-exponential sums]
{Generic Newton polygons for $L$-functions of $(A,B)$-exponential sums}
\begin{abstract}
In this paper, we consider the following $(A, B)$-polynomial $f$ over finite field:
\begin{equation*}
f(x_0,x_1,\cdots,x_n)=x_0^Ah(x_1,\cdots,x_n)+g(x_1,\cdots,x_n)+P_B(1/x_0),
\end{equation*}
where $h$ is a Deligne polynomial of degree $d$, $g$ is an arbitrary polynomial of degree $< dB/(A+B)$ and $P_B(y)$ is a one-variable polynomial of degree $\le B$. Let $\Delta$ be the Newton polyhedron of $f$ at infinity.
We show that $\Delta$ is generically ordinary if $p\equiv 1 \mod D$,
where $D$ is a constant only determined by $\Delta$.
In other words, we prove that the Adolphson--Sperber conjecture is true for $\Delta$.
\end{abstract}

\author[L.P. Yang]{Liping Yang}
\address{School of Mathematical Sciences, Capital Normal University, Beijing 100048, P.R. China}
\email{yanglp2013@126.com}
\author[H.Zhang]{Hao Zhang}
\address{Shing-Tung Yau Center of Southeast University, Yifu Architecture Building, No.2, Sipailou, Nanjing 210096, P.R. China}
\email{zhanhgao@126.com}

\keywords{$(A,B)$-polynomial, $L$-function, Exponential sum, Newton polygon, Hodge polygon.}
\subjclass[2010]{Primary 11T06, 11T23, 11S40. }
\maketitle

\section{Introduction}

Let $p$ be a prime number, $q$ be a power of $p$, and $\mathbb{F}_q$ be the finite field of $q$ elements.
For each positive integer $k$, let $\mathbb{F}_{q^k}$ be the finite extension of $\mathbb{F}_q$ of degree $k$ and $\mathbb{F}_{q^{k}}^{*}$ denote the set of non-zero elements in $\mathbb{F}_{q^{k}}$. Given a Laurent polynomial $f(x)\in \mathbb{F}_q[x_1^{\pm},\cdots ,x_n^{\pm}]$, we define the toric exponential sum as
$$S_k^*(f):=\sum_{x\in (\mathbb{F}_{q^{k}}^{*})^{n}} \zeta_p^{{\rm Tr}_k f(x)},$$
where ${\rm Tr}_k: \mathbb{F}_{q^k}\rightarrow \mathbb{F}_p$ denotes the trace of the field extension.
A well known theorem of Dwork--Bombieri--Grothendieck implies that the following generating $L$-function is rational:
$$L^*(f,T):=\exp\Big(\sum_{k=1}^{\infty} \frac{S_k^*(f)T^k}{k}\Big)=\frac{\prod_{i=1}^{d_1}(1-\alpha_i T)}{\prod_{j=1}^{d_2}(1-\beta_j T)},$$
where $\alpha_i (1\le i\le d_1)$ and $\beta_j (1\le j\le d_2)$ are non-zero algebraic integers. Equivalently, $$S_k^*(f)=\beta_1^k+\cdots+\beta_{d_2}^k-\alpha_1^k-\cdots-\alpha_{d_1}^k. $$
We are interested in studying the exponential sum $S_k^*(f)$, which is reduced to understanding the reciprocal zeros and reciprocal poles.
 However, one does not even know the number of zeros and poles
if there is no smooth condition on $f$.

Let us introduce a smooth condition first. Using multi-index notation, we write

$$f(x)=\sum_{\mathbf v\in A}a_{\mathbf v}x^{\mathbf{v}},$$
where $A$ is a finite subset of $\mathbb Z^n$ and $a_{\mathbf v}\in \mathbb{F}_q^{*}$. The convex hull of $A$ and the origin is called the {\it Newton polyhedron of $f$ at infinity},
denoted by $\Delta_{\infty} (f)$. If $\delta$ is a face of $\Delta_{\infty}(f)$, we define the restriction of $f$ to $\delta$
to be $$f_{\delta}=\sum_{\mathbf{v}\in \delta}a_{\mathbf v}x^{\mathbf{v}}.$$
Let $\mathbb{\bar{F}}_q^{*}$ denote the nonzero elements of an algebraic closure of $\mathbb F_q$.
\begin{defn}
The Laurent polynomial $f$ is called {\it $\Delta_{\infty}(f)$-nondegenerate} if for each closed face $\delta$ of $\Delta_{\infty} (f)$
 which does not contain the origin, the system of equations
$$x_1\frac{\partial f_{\delta}}{\partial x_1}=\cdots=x_n \frac{\partial f_{\delta}}{\partial x_n}=0$$
has no common zeros in $\mathbb{\bar{F}}_q^{*}$.
\end{defn}

 We always assume that $\Delta_{\infty}(f)$ is $n$-dimensional. If $f$ is $\Delta_{\infty}(f)$-nondegenerate, then
 Adolphson and Sperber \cite{AS89} proved that the $L$-function $L^*(f,T)^{(-1)^{n-1}}$ is a polynomial of degree $n!{\rm Vol}(\Delta_{\infty}(f))$, where ${\rm Vol}(\Delta_{\infty}(f))$ is the volume of $\Delta_{\infty}(f)$.
Using $l$-adic methods, Denef and Loeser \cite{DL} explicitly determined the complex absolute values of the zeros of $L^*(f,T)^{(-1)^{n-1}}$.
Hence the remaining interesting question is to determine the $p$-adic absolute values of the zeros.
This question equals to determine the $q$-adic Newton polygon of the polynomial
$$L^*(f,T)^{(-1)^{n-1}}=\sum_{i=0}^{n!{\rm Vol}(\Delta_{\infty}(f))}A_i(f)T^i, A_i(f)\in \mathbb{Z}[\zeta_p].$$
The $q$-adic {\it Newton polygon} of $L^*(f,T)^{(-1)^{n-1}}$, denoted by ${\rm NP}(f)$, is defined as the lower convex  closure in $\mathbb{R}^2$ of the points $$(i,{\rm ord}_q A_r(f)),\ i=0,1,...,n!{\rm Vol}(\Delta_{\infty}(f)).$$

Unfortunately, there is no effective way to compute the Newton polygon since the Newton polygon varies greatly as $f$ and $p$ vary.
However, when $f$ is $\Delta_{\infty}(f)$-nondegenerate, Adolphson and Sperber \cite{AS89} proved that
${\rm NP}(f)$ has a certain combinatorial lower bound, called the Hodge polygon (see Definition \ref{def1} below).

Let $\Delta$ be the $n$-dimensional integral convex polyhedron $\Delta_{\infty}(f)$ in $\mathbb{R}^n$,
and $C(\Delta)$ be the cone in $\mathbb{R}^n$ generated by $\Delta$.
 For a point $u\in \mathbb{R}^n$, the weight $w(u)$ is defined
 to be the smallest non-negative real number $c$ such that $u\in c\Delta$.
 If there is no such $c$, then we define $w(u)=\infty$.
Let $\delta$ be a co-dimensional $1$ face of $\Delta$ which does not contain the origin. Let $C(\delta)$ be the cone in $\mathbb R^n$ over $\delta$.
Let $\sum_{i=1}^ne_iX_i=1$ be the equation of the hyperplane $\delta$ in $\mathbb{R}^n$, where $e_i$ ($i=1,\cdots,n$) are rational numbers.
Then for $u=(u_1,\cdots,u_n)\in C(\delta)$, one has $w(u)=\sum_{i=1}^ne_iu_i$.
Let $D(\delta)$ be the least common denominator of the rational numbers $e_i$ $(1\le i\le  n)$, and
$D(\Delta)$ be the least common multiple of all the $D(\delta)$ with $\delta$ running over all
the co-dimensional $1$ faces of $\Delta$ which do not contain the origin.
Then $w(\mathbb{Z}^n)\subseteq \frac{1}{D(\Delta)}\mathbb{Z}_{\ge 0}\cup \{\infty\},$
where $\mathbb{Z}_{\ge 0}$ denotes the set of nonnegative integers.
The integer $D:=D(\Delta)$ is called the denominator of $\Delta$.

For an integer $k$, let
$$W_{\Delta}(k):={\rm card} \Big\{u\in \mathbb{Z}^n: w(u)=\frac{k}{D}\Big\}$$
be the number of lattice points in $\mathbb{Z}^n$ with weight $k/D$.
Let $$H_{\Delta}(k):=\sum_{i=0}^n(-1)^i \binom{n}{i}W_{\Delta}(k-iD).$$

\begin{defn}\label{def1}
The {\it Hodge polygon} ${\rm HP}(\Delta)$ of $\Delta$ is defined to be the lower convex polygon in $\mathbb{R}^2$
with the origin and vertices
$$\Big(\sum_{k=0}^{m}H_{\Delta}(k), \frac{1}{D}\sum_{k=0}^mk H_{\Delta}(k)\Big),\quad m=0,1,2,\cdots,nD.$$
\end{defn}

In this paper, we are interested in the variation of ${\rm NP}(f)$ when $f$ varies.
Let $\mathcal{M}_p(\Delta)$
be the set of $\Delta$-nondegenerate $f$ over the algebraic closure of $\mathbb{F}_p$ with $\Delta_{\infty}(f)=\Delta$.
By the Grothendieck specialization theorem, the lowest Newton polygon
$${\rm GNP}(\Delta,p):=\inf_{f\in \mathcal{M}_p(\Delta)}{\rm NP}(f)$$
exists, and can be obtained for all $f$ in some Zariski open dense subset of $\mathcal{M}_p(\Delta)$.
The lowest polygon is called the {\it generic Newton polygon}, denoted by ${\rm GNP}(\Delta,p)$.

It follows from Adolphson and Sperber \cite{AS89} that for $f\in \mathcal{M}_p(\Delta)$, one has ${\rm NP}(f)\ge {\rm HP}(\Delta)$.
For every prime $p$ and every $f\in \mathcal{M}_p(\Delta)$, we have
$${\rm NP}(f)\ge {\rm GNP}(\Delta, p)\ge {\rm HP}(\Delta).$$
One intriguing question is to decide when the generic Newton polygon coincides with the Hodge polygon.
The family $\mathcal{M}_p(\Delta)$ is called {\it generically ordinary} if ${\rm GNP}(\Delta, p)= {\rm HP}(\Delta)$. And a Laurent polynomial is called {\it ordinary} if ${\rm NP}(f)={\rm HP}(\Delta).$
Based on results of Sperber \cite{SP1,SP2,SP3} and Carpentier \cite{MC} on $p$-adic estimates of exponential sums, Adolphson and Sperber \cite{AS89} conjectured that
\begin{con}(Adolphson--Sperber) If $p\equiv 1 \pmod D$, then ${\rm GNP}(\Delta,p)={\rm HP}(\Delta)$.
\end{con}
This conjecture generalizes a conjecture of Dwork \cite{DW73} and Mazur \cite{MA}. Wan \cite{WD1,WD2} established several decomposition theorems for the generic Newton polygon and proved that the Adolphson--Sperber conjecture is true for $n\le 3$ but false for every higher dimension $n\ge 4$. He also proved that a weaker form of the A-S conjecture is true when we replace $D$ by another effectively computable positive integer $D^\ast$.

Le \cite{Le} generalized Wan's theory by giving the regular decomposition theorem and showing that the star decomposition, hyperplane decomposition and collapsing decomposition raised by Wan are  regular decompositions.
 In \cite{Le}, Le claimed the following result without proof: Let $\Delta_{\infty}(f)=\bigcup_i\Delta_i$ be a complete regular integral decomposition of $\Delta_{\infty}(f)$ and $f_i$ be the restriction of $f$ on $\Delta_i$. If each $f_i$ is generically $\Delta_i$-nondegenerate, then $f$ is also generically $\Delta_{\infty}(f)$-nondegenerate.
 It seems that this result is not evident. In Section 2, we will give the explicit proof by using some results of Gelfand, Kapranov and Zelevinsky.

A polynomial $h\in \mathbb{F}_q[x_1,\cdots,x_n]$ of degree $d>0$ is called a {\it Deligne polynomial} if its degree $d$ prime to $p$, and
its highest degree term defines a smooth projective hypersurface in  $\mathbb{P}^{n-1}$.
Let $A, B, d$ be positive integers relatively prime to $p$ and $f$ be an $(A, B)$-polynomial of the following form:
\begin{equation}\label{eqn27}
f(x_0,x_1,\cdots,x_n)=x_0^Ah(x_1,\cdots,x_n)+g(x_1,\cdots,x_n)+P_B(1/x_0)\in \mathbb{F}_q[x_0^{\pm},x_1,\ldots,x_n],
\end{equation}
where $h$ is a Deligne polynomial of degree $d$, $g$ is an arbitrary polynomial of degree $< dB/(A+B)$ and $P_B(y)$ is a one-variable polynomial of degree $\le B$. Set $\Delta=\Delta_{\infty}(f)$.
Let $\widetilde{\mathcal{M}}_p(\Delta)$ be the set of nondegenerate $(A,B)$-polynomials of form (\ref{eqn27})
and
$$\widetilde{{\rm GNP}}(\Delta,p):=\inf_{f\in \widetilde{\mathcal{M}}_p(\Delta)}{\rm NP}(f).$$
 Fu and Wan \cite{FW} proved that
if $p\equiv 1 \mod [A,dB]$, then $\widetilde{{\rm GNP}}(\Delta,p)={\rm HP}(\Delta)$.
In \cite[Remark 5.3]{FW}, Fu and Wan calculated that  $D(\Delta)=[A,B,dB/(A+B,dB)]$ and
conjectured that the Adolphson--Sperber conjecture is true for $\Delta$.
For $A=B=1$, Le \cite{Le} conjectured that
\begin{con} \label{con1} $p\equiv 1 \pmod D$ if and only if $\widetilde{{\rm GNP}}(\Delta,p)={\rm HP}(\Delta)$.
\end{con}
Moreover, Le proved that Conjecture \ref{con1} holds for odd $d$. Recently, Li \cite{Li} gave a counterexample of Conjecture \ref{con1} for even $d$.
Since the family considered by Le is slightly smaller than $\mathcal{M}_p(\Delta)$,
 Li \cite{Li} asked whether Conjecture \ref{con1} is true if the family $f$ becomes larger, i.e., with $\deg g \le\frac{d}{2}$ instead of $\deg g <\frac{d}{2}$. For such $f$, we have $\widetilde{\mathcal{M}}_p(\Delta)= \mathcal{M}_p(\Delta)$.
Hence $\widetilde{{\rm GNP}}(\Delta,p)={\rm GNP}(\Delta,p)$. It is well known that if ${\rm GNP}(\Delta,p)={\rm HP}(\Delta)$, then $p\equiv 1\pmod D$ (see \cite[Section 1.3]{WD2}).
Hence Li's question equals to ask that whether the Adolphson--Sperber conjecture holds for $\Delta$ when $A=B=1$.

Using Wan's facial decomposition theorem (Theorem \ref{thm8} below) and the regular decomposition theorem (Corollary \ref{thm21} below) generalized by Le, we prove the following:
\begin{thm}\label{thm1}
Suppose $f$ is given by (\ref{eqn27}) and $\Delta=\Delta_{\infty}(f)$. Then the Adolphson--Sperber conjecture is true for $\Delta$, i.e., if $p\equiv 1 \mod[A,B,dB/(A+B,dB)]$, then ${\rm GNP}(\Delta,p)={\rm HP}(\Delta)$.
\end{thm}

This paper is organized as follows. In Section 2, we will introduce Wan's decomposition theory and the regular decomposition theorem generalized by Le. In Section 3, we present the proof of Theorem \ref{thm1}.

\section{Decomposition theory}

Essentially, the idea of decomposition theory is to reduce from
the harder non-diagonal case to the easier diagonal one. So in this section we introduce the diagonal local theory first.
A more complete treatment of the theory can be found in \cite{WD2}.

A Laurent polynomial $f\in \mathbb{F}_q[x_1^{\pm},\cdots,x_n^{\pm}]$ is called {\it diagonal} if $f$ has exactly $n$ non-constant terms and $\Delta:=\Delta_{\infty}(f)$ is $n$-dimensional.
Write $f(x)=\sum_{j=1}^na_jx^{\mathbf v_j}$, where $a_j\in \mathbb{F}_q^{*}$ and each $\mathbf v_j$ is written as a column vector.
Then the square matrix of $\Delta$ is defined to be
$$\mathbf{M}(\Delta)=(\mathbf v_1,\cdots,\mathbf v_n).$$
 If $f$ is diagonal, then $\det\mathbf{M}(\Delta)\neq0$ in $\mathbb Z$. We have
\begin{prop}\label{diagonal}\cite[Section 2.1]{WD2}
$f$ is $\Delta_{\infty}(f)$-nondegenerate if and only if
$\gcd(p,\det \mathbf{M}(\Delta))=1$.
\end{prop}

Let $S(\Delta)$ be the solution set of the following linear system
$$\mathbf{M}(\Delta)\cdot(r_1,r_2,...,r_n)^t \equiv 0 \pmod 1,\quad r_i\in \mathbb{Q}\cap [0,1),$$
where $(r_1,r_2,...,r_n)^t$ means the transpose of $(r_1,r_2,...,r_n)$.
Then $S(\Delta)$ is an abelian group and its order is given by $|\det \mathbf{M}(\Delta)|$.
By the fundamental structure of finite abelian group, $S(\Delta)$ can be decomposed into a direct product of invariant factors, i.e.,
$$S(\Delta)=\bigoplus_{i=1}^n \mathbb{Z}/s_i(\Delta)\mathbb{Z},$$
where $s_i(\Delta)|s_{i+1}(\Delta)$ for $i=1,2,...,n-1$.
Wan \cite{WD2} proved the following criterion for ordinarity.
\begin{prop}\label{prop2} Suppose that  $f\in \mathbb{F}_q[x_1^{\pm},\cdots,x_n^{\pm}]$ is  $\Delta_{\infty}(f)$-nondegenerate and diagonal. Let $s_n(\Delta)$ be the largest invariant factor of $S(\Delta)$. If $p\equiv 1 \mod s_n(\Delta)$, then $f$ is ordinary.
\end{prop}
Note that $s_n(\Delta)|\det\mathbf{M}(\Delta)$. We conclude that if $p\equiv 1 \mod (\det\mathbf{M}(\Delta))$, then $f$ is ordinary.

 For the general case, we will describe the facial decomposition for the Newton polygon \cite{WD1} and the regular
decomposition theorem for the generic Newton polygon \cite{Le}.

\begin{thm}\label{thm8}\cite[Facial decomposition theorem]{WD1} Let $f$ be a Laurent polynomial with $n$ variables over $\mathbb{F}_q$.
Assume $\Delta_{\infty}(f)$ is $n$-dimensional and $\delta_1,\cdots,\delta_h$ are all the co-dimensional $1$ faces of $\Delta_{\infty}(f)$
which do not contain the origin.
For $1\le i\le h$, let $\Delta_i$ be the convex polytope of lattice points in $\delta_i$ and the origin. Then $f$ is $\Delta_{\infty}(f)$-nondegenerate and ordinary if and only if $f_{\delta_i}$ is $\Delta_i$-nondegenerate and ordinary
for $1\le i\le h$.
\end{thm}

By facial decomposition theorem, to study the ordinarity of $f$, we may reduce to the case when $\Delta_\infty(f)$ has only one co-dimensional $1$ face $\delta$ not containing the origin. A decomposition of $\delta$ then gives a decomposition of $\Delta_\infty(f)$ by connecting with the origin.

\begin{defn}
Let $\delta$ be an $(n-1)$-dimensional convex polytope with $A$ being a finite subset of $\delta$
including all vertices of $\delta$. A {\it subdivision} of $(\delta, A)$ is a family
$\mathcal{F}=\{(\delta_i, A_i):i\in I\}$ of polytopes such that
\begin{itemize}
  \item each $A_i$ is a subset of $A\cap\delta_i$, $A_i$ contains all vertices of $\delta_i$ and $\dim \delta_i=n-1.$
  \item any $\delta_i\cap \delta_j$ is a face (possibly empty) of both $\delta_i$ and $\delta_j$,
  and $$A_i\cap (\delta_i\cap \delta_j)=A_j\cap (\delta_i\cap \delta_j). $$
  \item the union of all $\delta_i$ coincides with $\delta$.
\end{itemize}
\end{defn}
Each $\delta_i$ is called a cell of the subdivision $\mathcal{F}$.
In particular, if each $\delta_i$ is a simplex, then we call that $\mathcal{F}=\{(\delta_i, A_i):i\in I\}$
is a {\it triangulation} of $(\delta,A)$.
We say that a cell $\delta_i$ is \emph{indecomposable} if it is a simplex and contains no point in $A$ other than vertices.
An integral polytope is a convex polytope whose vertices are lattice points.
A subdivision $\mathcal{F}$ is called {\it integral} if each $\delta_i$ is an integral polytope. An integral subdivision
$\mathcal{F}$ is called \emph{complete} if each $\delta_i$ is indecomposable. Clearly, a complete
subdivision should be a triangulation.

In \cite{Le}, Le studied the regular subdivision which is defined as follows.
\begin{defn}(\cite[Section 4.2.2]{Le} or \cite[Section 10]{WD3})
A subdivision $\mathcal{F}=\{(\delta_i, A_i):i\in I\}$ of $(\delta,A)$ is called  \emph{regular} if
there is a piecewise linear function $\phi:\delta\to \mathbb{R}$ such that
\begin{enumerate}
  \item $\phi$ is concave, i.e. $\phi(\lambda x+(1-\lambda)x')\ge \lambda\phi(x)+(1-\lambda)\phi(x')$, for all $x,x'\in \delta$ and $0\le \lambda\le 1$;
    \item the domains of linearity of $\phi$ are precisely the $\delta_i$ for $i\in I$.
\end{enumerate}
\end{defn}

\begin{defn}\cite[Chapter 2]{DRS}
 Let $\mathcal{F}=\{(\delta_i, A_i):i\in I\}$ and $\mathcal{F}'=\{(\delta_j', A_j'): j\in J\}$
 be two subdivisions of $(\delta,A)$. We say that $\mathcal{F}$ is a {\it refinement} (resp. {\it regular refinement}) of $\mathcal{F}'$
 if the collection of $(\delta_i,A_i)$ such that $\delta_i\subseteq \delta_j'$ forms a subdivision (resp. regular subdivision) of $(\delta_j',A_j')$ for any $j\in J$.
\end{defn}
We have the following property for regular subdivision.

\begin{lem}\label{lem4}\cite[Lemma 2.3.16]{DRS}
Let $\mathcal{F}$ and $\mathcal{F}'$ be two subdivisions of $(\delta,A)$ such that
$\mathcal{F}$ is regular and $\mathcal{F}'$
is a regular refinement of $\mathcal{F}$. Then $\mathcal{F}'$ is also regular.
\end{lem}

In the following, let $B$ be a set of lattice points in $\mathbb R^{n-1}$ and let $g=\sum_{\mathbf w\in B}b_{\mathbf w}y^{\mathbf w}\in \mathbb F_q[y_1^{\pm},\ldots,y_{n-1}^{\pm}]$ be the generic Laurent polynomial associated to $B$. The Newton polyhedron $\Delta(g)$ of $g$ is defined to be the convex hull of $B$ in $\mathbb R^{n-1}$. We assume $\text{dim }\Delta(g)=n-1$.
The {\it principle $B$-determinant} $E_{B}(g)$ of $g$ is defined to be the resultant of
$$\Big\{y_1\frac{\partial g}{\partial y_1},\cdots,y_{n-1}\frac{\partial g}{\partial y_{n-1}},g\Big\}.$$
Recall that (see \cite[Chapter 10]{GKZ}) $E_{B}(g)$ is a polynomial in variables
$\{b_{\mathbf w}\}_{\mathbf w\in B}$.
Let
$$E_{B}(g)=\sum_{k\in \mathbb{Z}^{|B|}}c_k\mathbf b^k,$$
where $|B|$ denotes the number of points in $B$ and $\mathbf b=(b_{\mathbf w})_{\mathbf w\in B}$. The Newton polyhedron of $E_{B}(g)$ is defined to be the convex hull of lattice points $k\in \mathbb{Z}^{|B|}$ such that $c_k\neq 0$. When $B=\{\mathbf v_1,\cdots,\mathbf v_n\}$ is the set of vertices of an $(n-1)$-dimensional simplex in $\mathbb R^{n-1}$ and $g=\sum_{i=1}^n b_iy^{\mathbf v_i}.$
Then
$$E_{B}(g)=\pm\big({\widetilde{\rm Vol}}(\Delta(g))b_1\cdots b_n\big)^{\widetilde{\rm Vol}(\Delta(g))},$$ where $\widetilde{\rm Vol}(\Delta(g))$ is $(n-1)!$ times the volume of $\Delta(g)$.
\begin{defn}
The Laurent polynomial $g$ is called {\it $\Delta(g)$-nondegenerate} if for each closed face $\delta$ of $\Delta (g)$, the system of polynomial equations
$$g_{\delta}=y_1\frac{\partial g_{\delta}}{\partial y_1}=...=y_{n-1}\frac{\partial g_{\delta}}{\partial y_{n-1}}=0$$
have no solutions in $\mathbb{\bar{F}}_q^{*},$
where $g_\delta=\sum_{\mathbf w\in B\cap \delta}b_{\mathbf w}y^{\mathbf w}.$
\end{defn}
One has the following explicit criterion to determine the $\Delta(g)$-nondegeneracy.
\begin{thm}\label{thm23} \cite[Chapter 10]{GKZ} The generic Laurent polynomial $g$ is  $\Delta(g)$-nondegenerate if and only if $E_{B}(g)\neq 0$.
\end{thm}
Gelfand, Kapranov and Zelevinsky showed the following result:
\begin{thm}\label{thm24}\cite[Chapter 10]{GKZ}
There is a one-to-one correspondence between the set of all regular triangulations of $(\Delta(g),B)$ and the set of vertices of the Newton polyhedron of $E_{B}(g)$.
If $\mathcal{F}=\{(\delta_i, B_i):i\in I\}$ is a complete regular triangulation of $(\Delta(g),B)$, then the corresponding term in $E_{B}(g)$ has the form
$$ \prod_{\delta_i}E_{B_i}(g_{\delta_i}),$$
where  $g_{\delta_i}=\sum_{\mathbf w\in B_i}b_{\mathbf w}  y^{\mathbf w}$ is the restriction of $g$ to $(\delta_i, B_i)$.
\end{thm}

Let $\Delta$ be an $n$-dimensional, integral, convex polytope in $\mathbb R^n$ which has only one co-dimensional $1$ face $\delta$ not containing the origin.
Let $A$ be the set of all the  lattice points on face $\delta$,
$$f(x)=\sum_{{\mathbf v}\in A}a_{\mathbf v}x^{\mathbf v}\in \mathbb F_q[x_1^{\pm }, x_2^{\pm },\cdots, x_n^{\pm }]$$
be the generic Laurent polynomial associated to $A$. We have $\Delta_\infty(f)=\Delta$. For a subdivision $\mathcal{F}=\{(\delta_i, A_i):i\in I\}$ of $\delta$, we always assume $A_i=\delta_i\cap A$. Thus we say $\{\delta_i\}_{i\in I}$ is a subdivision of $\delta$ for simplicity.

One can choose an affine transformation $T\in \text{GL}_n(\mathbb Z)$ on $\mathbb R^n$ such that
$T(A)$ lies on the hyperplane $v_n=e$ for some positive integer $e$. Let
$$f'(x):=\sum_{{\mathbf v}\in A}a_{\mathbf v} x^{T({\mathbf v})}=x_n^e g'(x_1,\ldots,x_{n-1}),$$
where $g'\in \mathbb F_q[x_1^{\pm},\ldots, x_{n-1}^{\pm}]$.
Then $f$ is $\Delta_\infty(f)$-nondegenerate if and only if $f'$ is $\Delta_\infty(f')$-nondegenerate.
When $p\nmid e$, one checks by definition that $f'$ is $\Delta_\infty(f')$-nondegenerate if and only if $g'$ is $\Delta(g')$-nondegenerate.

Let $\{\delta_i\}_{i\in I}$ be a complete regular triangulation of $\delta$. Under the affine transformation $T$, one derives a complete regular triangulation $\{\delta'_i\}_{i\in I}$ (resp. $\{\delta''_i\}_{i\in I}$) of $T(\delta)$ (resp. $\Delta(g')$).
Denote the restriction of $f$ (resp. $f'$, resp. $g'$) to $\delta_i$ (resp. $\delta'_i$, resp. $\delta''_i$) by $f_{i}$ (resp. $f'_{i}$, resp. $g'_{i}$) . Denote   the set of lattice points in $\Delta(g')$ by $B'$. Let $B'_i:=B'\cap \delta''_i$.

If each $f_{i}$ is $\Delta_{\infty}(f_{i})$-nondegenerate, then each $f'_{i}$ is $\Delta_{\infty}(f'_{i})$-nondegenerate.
Since $\delta'_i$ is indecomposable, one has $f'_{i}$ is diagonal. By Proposition \ref{diagonal}, one deduces $p\nmid e$. Hence $g_{i}'$ is $\Delta(g_{i}')$-nondegenerate. By Theorem \ref{thm23}, one has $E_{B'_i}(g'_{i})\ne 0$. It follows from Theorem \ref{thm24} that $E_{B'}(g')\ne 0$. That is, $g'$ is $\Delta(g')$-nondegenerate. Hence $f$ (resp. $f'$) is $\Delta_{\infty}(f)$ (resp. $\Delta_{\infty}(f')$)-nondegenerate.  So we obtain:
\begin{cor}\label{cor1}
If $\mathcal{F}=\{(\delta_i, A_i):i\in I\}$ is a complete regular triangulation of $(\Delta(f),A)$ such that each $f_{i}$ is $\Delta_{\infty}(f_{i})$-nondegenerate, then $f$ is
$\Delta_{\infty}(f)$-nondegenerate.
\end{cor}

Let $\phi:\delta \rightarrow \mathbb{R}$ be a continuous function.
We extend the domain of $\phi$ to $C(\Delta)$ by setting
$$\phi(\mathbf{r})=w(\mathbf{r})\phi\Big(\frac{\mathbf{r}}{w(\mathbf{r})}\Big).$$

\begin{defn} For $\mathbf{r}\in C(\Delta)$, we define
$$m(\phi,A;\mathbf{r}):=\sup \Big\{\sum_{\mathbf v\in A} l_{\mathbf v}\phi (\mathbf{v}): \sum_{\mathbf v\in A}l_{\mathbf v}\mathbf{v}=\mathbf{r}, l_{\mathbf v}\in \mathbb R_{\geq 0} \Big\}.$$
We call $\phi$ is homogeneous with respect to $A$ if
$$m(\phi,A;\mathbf{r})=\inf \Big\{\sum_{\mathbf v\in A} l_{\mathbf v}\phi (\mathbf{v}): \sum_{\mathbf v\in A}l_{\mathbf v}\mathbf{v}=\mathbf{r}, l_{\mathbf v}\in \mathbb R_{\geq 0} \Big\}$$
for all $\mathbf{r}\in C(\Delta)$.
\end{defn}

Let $\{\delta_i\}_{i\in I}$ be a regular integral subdivision of $\delta$ and $\phi$ be a piecewise linear function on $\delta$ corresponding to $\{\delta_i\}_{i\in I}$. Let $A_i=A\cap \delta_i$. One has that
$$\phi(\lambda \mathbf{u}+(1-\lambda)\mathbf{u}')=\lambda\phi(\mathbf{u})+(1-\lambda)\phi(\mathbf{u}')\ {\rm for}\ \mathbf{u},\mathbf{u}'\in \delta_i, 0\le \lambda \le 1.$$
For $i\in I$, denote by $C(\delta_i)$ the cone generated by $\delta_i$. For $\mathbf{r}, \mathbf{r'}\in C(\delta_i)$, $a\in \mathbb R_{\geq 0}$, using the property of weight function that
$w(\mathbf{r}+\mathbf{r'})=w(\mathbf{r})+w(\mathbf{r'})$, $w(a\mathbf r)=aw(\mathbf r)$, one has
\begin{align*}
\phi(\mathbf{r}+\mathbf{r'})&=\phi\Big(\frac{\mathbf{r}+\mathbf{r'}}{w(\mathbf{r}+\mathbf{r'})}\Big)\cdot w(\mathbf{r}+\mathbf{r'})\\
&=\phi\Big(\frac{w(\mathbf{r})}{w(\mathbf{r}+\mathbf{r'})}\frac{\mathbf{r}}{w(\mathbf{r})}+
\frac{w(\mathbf{r'})}{w(\mathbf{r}+\mathbf{r'})}\frac{\mathbf{r'}}{w(\mathbf{r'})}\Big)\cdot w(\mathbf{r}+\mathbf{r'})\\
&=w(\mathbf{r})\phi\Big(\frac{\mathbf{r}}{w(\mathbf{r})}\Big)+w(\mathbf{r'})\phi\Big(\frac{\mathbf{r'}}{w(\mathbf{r'})}\Big)
=\phi(\mathbf{r})+\phi(\mathbf{r'}),\\
\phi(a\mathbf r)&=w(a\mathbf r)\phi\Big(\frac{a\mathbf r}{w(a\mathbf r)}\Big)=aw(r)\phi\Big(\frac{\mathbf r}{w(\mathbf r)}\Big)=a\phi(\mathbf r).
\end{align*}
We conclude that $C(\delta_i)$($i\in I$) are precisely the domains of linearity of $\phi$.
Hence $m(\phi,A_i;\mathbf{r})=\phi(\mathbf{r})$ for any $\mathbf{r}\in C(\delta_i)$.
That is, $\phi|_{\delta_i}$ is homogeneous with respect to $A_i$.

Combining with \cite[Proposition 4.9, Theorem 5.1]{Le}, we obtain that
\begin{thm}\label{thm25}
Let $\{\delta_i\}_{i\in I}$ be a regular integral triangulatin of $\delta$. If $f$ is generically $\Delta_\infty(f)$-nondegenerate, $f_{i}$ is generically $\Delta_\infty(f_{i})$-nondegenerate and ordinary for each $i$, then $f$ is generically ordinary.
\end{thm}
By Corollary \ref{cor1} and Theorem \ref{thm25}, we have
\begin{cor}
\label{thm21}(Regular subdivision theorem, \cite[Theorem 4.3]{Le}) Let $\{\delta_i\}_{i\in I}$ be a complete regular integral triangulatin of $\delta$.
If each $f_{i}$ is generically $\Delta_\infty(f_{i})$-nondegenerate and ordinary, then $f$ is also generically $\Delta_\infty(f)$-nondegenerate
 and ordinary.
\end{cor}

Let $H$ be a hyperplane passing through the origin such that the intersection of $\delta$ and $H$ is an integral polytope of co-dimension $2$. This hyperplane cuts the polytope $\Delta$ (resp. the face $\delta$) into two pieces. In \cite[Section 7]{WD1}, such decomposition of $\Delta$ is called a {\it hyperplane decomposition} .
In Section 3, we will use a series of hyperplane decompositions to obtain a complete regular integral triangulation of $(A,B)$-polytope.

\section{Proof of Theorem \ref{thm1}}

In this section, we will give a complete regular integral triangulation of $\Delta$ associated to the $(A,B)$-polynomials. In the end, using Corollary \ref{thm21},
we prove that Theorem \ref{thm1} is true.

Let $f$ be an $(A,B)$-polynomial of the form  (\ref{eqn27}).
Let $\mathbf{e}'_0,\mathbf e'_1,\cdots,\mathbf e'_n$ be the standard basis for $\mathbb{R}^{n+1}$.
Then $\Delta=\Delta_{\infty}(f)$ is the polytope spanned by the origin and the vectors $-B\mathbf e'_0,A\mathbf e'_0,A\mathbf e'_0+d\mathbf e'_1,\cdots,A\mathbf e'_0+d\mathbf e'_n$.

There are two co-dimensional $1$ faces of $\Delta$ not containing the origin:
the face $\delta_d$ spanned by $-B\mathbf e'_0,A\mathbf e'_0+d\mathbf e'_1,\cdots,A\mathbf e'_0+d\mathbf e'_n$
and the face $\delta'_d$ spanned by $A\mathbf e'_0,A\mathbf e'_0+d\mathbf e'_1,\cdots,A\mathbf e'_0+d\mathbf e'_n$.
The facial decomposition allows us to study the nondegeneracy and ordinarity on the faces $\delta_d$
 and $\delta_d'$.
 Let $\Delta_d$ (resp. $\Delta_d'$) be the convex hull of $\delta_d\cup 0$ (resp. $\delta_d'\cup 0$) in $\mathbb R^{n+1}$.

 It follows from \cite[Section 5]{FW} that
 \begin{eqnarray*}
D(\Delta_d')&=&A
\end{eqnarray*}
and
\begin{eqnarray*}
D(\Delta_d)&=&[B,dB/(A+B,dB)]=dB/(A+B,d).
\end{eqnarray*}
Let $s:=(A+B,d)$. If $s=1$, then $D(\Delta_d)=dB$. It follows from \cite[Theorem 5.1]{FW} that the Adolphson--Sperber conjecture
is true for  $\Delta$.
Hence we always assume that $s\ge 2$ in the following.
Wan \cite[Section 7]{WD1} has proved the Adolphson--Sperber conjecture holds for $\Delta_d'$.
By the facial decomposition theorem, it remains to show the Adolphson--Sperber conjecture is also true for $\Delta_d$.

Let $\mathcal{V}$ be the set of all the lattice points on $\delta_d$.
For $t=0,1,\ldots,s$,
let $$\mathcal V_t:=\Big\{(v_0,v_1,\cdots,v_{n})\in \mathbb{Z}_{\ge 0}^{n+1}: v_0=\frac{(A+B)t}{s}-B,\\ \sum_{i=1}^nv_i=\frac{dt}{s} \Big\}.$$
Then
$$\mathcal{V}=\bigcup_{t=0}^{s} \mathcal V_t.$$
Clearly, $\mathcal V_0=\{-B\mathbf{e}'_0\} $.
For $t=1,\ldots,s-1$, let $\mathcal{H}_t$ be the hyperplane passing through the origin and $\mathcal V_t$. Then $\mathcal{H}_t\cap \delta_d$ is a convex polytope of co-dimension $2$ spanned by $\mathcal V_t$. This gives a hyperplane decomposition of
$\Delta_d$. One has that $\{\mathcal H_t\}_{t=1}^{s-1}$ cuts $\delta_d$ into $s$ closed pieces $\delta_d(1),...,\delta_d(s)$, where  $\delta_d(t)$ is the convex hull of of $\mathcal V_{t-1}\cup\mathcal V_t$ for $t=1,\ldots, s-1$. Let $\phi_t$ be a concave piecewise linear function associated to $\mathcal H_t$. Then the function $\phi:=\sum_{t=1}^{s-1}\phi_t$ is concave and the domains of linearity of $\phi$ are exactly $\delta_d(t)$ for $t=1,\ldots,s$. It follows from that \cite[Example 2.2 in Chapter 7]{GKZ} that  $\{\mathcal H_t\}_{t=1}^{s-1}$ gives a regular subdivision of $\delta_d$:
\begin{eqnarray}\label{big}
\delta_d=\bigcup_{t=1}^s\delta_d(t).
\end{eqnarray}

If $n=1$, then $\delta_d(t)$ is a line segment containing only two lattice points $((A+B)(t-1)-sB)/s,d(t-1)/s)$ and $((A+B)t-sB)/s,dt/s)$  for $t\in \{1,\ldots,s\}$. Thus
$\delta_d=\bigcup_{t=1}^s\delta_d(t)$
is a complete regular triangulation of $\delta_d$.
Moreover,
\begin{align*}
\begin{vmatrix}
\frac{(A+B)(t-1)-sB}{s} & \frac{(A+B)t-sB}{s} \\\\
\frac{d(t-1)}{s}& \frac{dt}{s} \\
\end{vmatrix}=\pm\frac{dB}{s}.
\end{align*}
By Proposition \ref{diagonal} and Corollary \ref{thm21}, the Adolphson--Sperber conjecture is true for $\Delta_d$.
We assume $n\ge 2$ in what follows.

\begin{prop}\label{lem3} The polytope
$\delta_d$ has a complete regular triangulation.
\end{prop}
\begin{proof}
Let $L_d\subseteq\mathbb R^{n+1}$ be the hyperplane passing through $\delta_d$, and $\rho: L_d\to \mathbb R^n$ be the projection of $L_d$ onto its last $n$ coordinates, which is linear and one-to-one. We will first give complete regular triangulations of $\delta_d(1),\ldots,\delta_d(s)$ separately.

For $t=1,\cdots,s$, denote by $\tilde{\delta}_d(t)$ the image of $\delta_d(t)$ under the projection $\rho$. Given a subdivision (resp. regular subdivision) $\{\delta_{i,t}\}_{i\in I}$ of $\tilde{\delta}_d(t)$, if each $\rho^{-1}(\delta_{i,t})$ is an integral polytope, then $\{\rho^{-1}(\delta_{i,t})\}_{i\in I}$ is a  subdivision (resp. regular subdivision) of $\delta_d(t)$. This is because that if  $\phi_t:\tilde{\delta}_d
(t) \mapsto \mathbb{R}$ is a piecewise linear function whose linearity domains are exactly $\{\delta_{i,t}\}_{i\in I}$, then define $\hat{\phi}_t: \delta_d(t)\to \mathbb{R}$ by  $\hat{\phi}_t(v_0,\mathbf{v})=\phi(\mathbf{v})$. The domains of linearity of $\hat \phi_t$ are exactly $\{\rho^{-1}(\delta_{i,t})\}_{i\in I}$. Moreover, if $\phi_t$ is concave, then $\hat\phi_t$ is concave.
It follows from \cite[Example 2.2 in Chapter 7]{GKZ} that $\{\rho^{-1}(\delta_{i,t})\}_{i\in I}$ is a subdivision (resp. regular subdivision) of $\delta_d(t)$. In the following, for subsets $X,Y\subset \mathbb R^n$, let $P(X,Y)$ denote the convex hull of $X\cup Y$ in $\mathbb R^n$, $\hat X$ denote the set $\rho^{-1}(X)\in L^d$, and $P(\hat X, \hat Y)$ denote the convex hull of $\hat X\cup \hat Y$ in $\mathbb R^{n+1}.$

For $t=0,1,\ldots,s$, define
$$\mathcal S_t:=\{(v_1,\cdots,v_n)\in \mathbb{R}_{\ge 0}^n: v_1+\cdots+v_n=dt/s\}.$$
Then $\tilde{\delta}_d(t)$ is the convex hull of $\mathcal S_t\cup\mathcal S_{t-1}$ for $t=1,\ldots,s$.
Lattice points in $\mathcal S_t$ are projections of elements in $\mathcal V_t$. Let $\mathbf{e}_0,\mathbf e_1,\cdots,\mathbf e_n$ be the standard basis for $\mathbb{R}^{n}$.
For $1\le i\le n-1$, we define
$$\Gamma_t^i:=\{(v_1,\cdots,v_n)\in \mathcal S_t :v_k=0\ {\rm for\ each }\ 1\le k\le i\}.$$
It is an $(n-i-1)$-dimensional simplex with vertices $\frac{dt}{s}\mathbf e_{i+1},\frac{dt}{s}\mathbf e_{i+2},\ldots,\frac{dt}{s}\mathbf e_{n}.$
We have a flag of $\mathcal S_t$:
$$ \Gamma_t^{n-1}\subseteq\cdots \subseteq \Gamma_t^1\subseteq \Gamma_t^0:=\mathcal S_t.$$

For $1\le j\le n-1$, we define
$$\Lambda_t^j:=\{(v_1,\cdots,v_n)\in \mathcal S_t : v_k=0\ {\rm for }\ n-j+1\le k\le n\}. $$
It is an $(n-j-1)$-dimensional simplex with vertices $\frac{d(t-1)}{s}\mathbf e_1, \frac{d(t-1)}{s}\mathbf e_2,\ldots,\frac{d(t-1)}{s}\mathbf e_{n-j}.$
We have a flag of $\mathcal S_t$:
$$\Lambda_t^{n-1}\subseteq \cdots \subseteq\Lambda_t^{1}\subseteq \Lambda_t^0:=\mathcal S_t.$$

First, we will give a regular subdivision of $\tilde\delta_d(t)=P(\mathcal S_t, \mathcal S_{t-1})$ for $1\le t\le s$, which will be lifted to a regular subdivision of $\delta_d(t)$. When $t=1$, we take the trivial subdivision $\{P(\Gamma_1^0, \Lambda_0^{n-1})\}$ of $\tilde\delta_d(1)$. When $t\geq 2$, for $1\le i\le n-1,$ let $\mathcal H_t^i$ be the affine subspace generated by $\Gamma_t^{i}\cup\Lambda_{t-1}^{n-i}$.
By inputting the coordinates of vertices of $\Gamma_t^i$ and $\Lambda_{t-1}^{n-i}$, we calculate that the equation of $\mathcal H_t^i$ is given by
\begin{equation}\label{hyp}
    \frac{s}{d(t-1)}\sum_{j=1}^ix_j+\frac{s}{dt}\sum_{j=i+1}^nx_j=1.
\end{equation}
\begin{lem}\label{hyperplane}
We have $\mathcal H_t^i\cap P(\mathcal{S}_t, \mathcal S_{t-1})=P(\Gamma_t^i, \Lambda_{t-1}^{n-i}).$
\end{lem}

\begin{proof}
By definition, we have $$P(\mathcal{S}_t, \mathcal S_{t-1})=\big\{\mathbf v=(v_1,\cdots, v_n)\in \mathbb R^n_{\geq 0}: \frac{d(t-1)}{s}\le v_1+\cdots+v_n\le \frac{dt}{s}\big\}.$$
Given $\mathbf x=(x_1,\cdots,x_n)\in P(\mathcal{S}_t, \mathcal S_{t-1})$ satisfying (\ref{hyp}), let
$$\lambda_j=\frac{sx_j}{d(t-1)}\ {\rm for}\ j=1,\ldots,i, \ {\rm and}\ \lambda_{j}=\frac{sx_{j}}{dt}\ {\rm for}\ j=i+1,\ldots,n.
$$
We have $\lambda_j\geq 0$ for $j=1,\cdots,n$, $\sum_{j=1}^n\lambda_j=1,$
and

$$\mathbf x=\lambda_1\frac{d(t-1)}{s}\mathbf e_1+\ldots+\lambda_i\frac{d(t-1)}{s}\mathbf e_i+\lambda_{i+1}\frac{dt}{s}\mathbf e_{i+1}+\ldots+\lambda_n\frac{dt}{s}\mathbf e_n.$$

Hence $\mathbf x\in P(\Gamma_t^i, \Lambda_{t-1}^{n-i}).$ So $\mathcal H_t^i\cap P(\mathcal{S}_t, \mathcal S_{t-1})\subseteq P(\Gamma_t^i, \Lambda_{t-1}^{n-i}).$ Another direction is obvious. This proves Lemma \ref{hyperplane}.
\end{proof}
It follows from Lemma \ref{hyperplane} that the intersection of $\mathcal H_t^i$ and $P(\mathcal{S}_t, \mathcal S_{t-1})$ is an integral polytope of codimensional $1$. Hence $\mathcal H_t^i$ gives a hyperplane decomposition of $P(\mathcal{S}_t, \mathcal S_{t-1})$.
Moreover, one verifies that
$\mathcal S_t$ and $\mathcal S_{t-1}$
lie on the opposite side of $\mathcal H_t^i$. Hence $\mathcal H_t^i$ cuts $P(\mathcal{S}_t, \mathcal S_{t-1})$ into two polytopes:
$$P(\mathcal S_t,\Gamma_t^i\cup\Lambda_{t-1}^{n-i})=P(\Gamma_t^0, \Lambda_{t-1}^{n-i}), \text{ and } P(\Gamma_t^i\cup\Lambda_{t-1}^{n-i}, \mathcal S_{t-1})=P(\Gamma^i_t, \Lambda_{t-1}^0).$$

Suppose $i\ge 2$ and the hyperplanes $\{\mathcal H_t^j\}_{j=1}^{i-1}$ cut $P(\mathcal S_t, \mathcal S_{t-1})$ into $i$ polytopes: $P(\Gamma_t^{i-1}, \Lambda_{t-1}^0)$ and  $P(\Gamma_t^j,\Lambda_{t-1}^{n-j-1})$ for $j=0,\ldots,i-2$.
Note that
$$P(\Gamma_t^i, \Lambda_{t-1}^{n-i})\subseteq \mathcal{H}_t^i\cap P(\Gamma^{i-1}_t, \Lambda_{t-1}^0)\subseteq \mathcal H_t^i\cap P( \mathcal{S}_t, \mathcal S_{t-1})=P(\Gamma_t^i, \Lambda_{t-1}^{n-i}).$$
Then $P(\Gamma_t^i, \Lambda_{t-1}^{n-i})= \mathcal{H}_t^i\cap P(\Gamma^{i-1}_t, \Lambda_{t-1}^0).$ Since $\Gamma_t^{i-1}$ and $\Lambda_{t-1}^0$ lie on the opposite side of $\mathcal H_t^i$, one has that $\mathcal H_t^i$ cuts $P(\Gamma_t^{i-1}, \Lambda_{t-1}^0)$ into two parts:
$$P(\Gamma_t^i\cup\Lambda_{t-1}^{n-i},\Gamma_t^{i-1})=P(\Gamma_t^{i-1}, \Lambda_{t-1}^{n-i}),\text{ and } P(\Gamma_t^i\cup\Lambda_{t-1}^{n-i},\Lambda_{t-1}^0)=P(\Gamma_t^i,\Lambda_{t-1}^0).$$
For $j=0,\ldots,i-2$, $P(\Gamma_t^j,\Lambda_{t-1}^{n-j-1})$ is contained in $P(\Gamma_t^0, \Lambda_{t-1}^{n-i})$. Then one has that the hyperplanes $\{\mathcal H_t^j\}_{j=1}^{i}$ cut $P(\mathcal S_t,\mathcal S_{t-1})$ into $i+1$ polytopes: $P(\Gamma_t^{i}, \Lambda_{t-1}^0)$ and  $P(\Gamma_t^j,\Lambda_{t-1}^{n-j-1})$ for $i=0,\ldots,i-1$. By induction we can prove the series of hyperplanes $\{\mathcal H_t^i\}_{i=1}^{n-1}$ divide $\tilde{\delta}_d(t)$ into $n$ parts:
\begin{equation}\label{eqn13}
\tilde{\delta}_d(t)=\bigcup_{i=0}^{n-1}P(\Gamma_t^i,\Lambda_{t-1}^{n-i-1}),
\end{equation}
which is a regular subdivision by the previous paragraph.

We now show that (\ref{eqn13}) can be lifted to a regular subdivision of $\delta_d(t)$.
For $i=0,1,\cdots,n-1$, let
$$\hat{\Gamma}_t^{i}:=\rho^{-1}(\Gamma_t^i)=\{((A+B)t/s-B,\mathbf{v})\in \mathbb{R}^{n+1}: \mathbf{v}\in \Gamma_t^{i}\}$$
and
$$\hat{\Lambda}_{t-1}^{n-i-1}:=\rho^{-1}(\Lambda^{n-i-1})=\{((A+B)(t-1)/s-B,\mathbf{v})\in \mathbb{R}^{n+1}: \mathbf{v}\in \Lambda_{t-1}^{n-i-1}\}.$$
Then $$P(\hat{\Gamma}_t^i,\hat{\Lambda}_{t-1}^{n-i-1})=\rho^{-1}\big(P(\Gamma_t^i,\Lambda_{t-1}^{n-i-1})\big)$$ is an integral polytope.
Hence we get a regular subdivision of $\delta_d(t)$:
 \begin{equation}\label{eqn6}
 \delta_{d}(t)=\bigcup_{i=0}^{n-1}P(\hat{\Gamma}_t^i,\hat{\Lambda}_{t-1}^{n-i-1}).
 \end{equation}

Next, for $1\le t\le s$ and $0\le i\le n-1$, we will give a regular triangulation for each  $P(\Gamma_t^i,\Lambda_{t-1}^{n-i-1})$, which will be lifted to a regular triangulation of $P(\hat{\Gamma}_t^i,\hat{\Lambda}_{t-1}^{n-i-1})$.
For $i+1\le k\le n, 1\le j\le \frac{td}{s}-1$,
let $\mathcal H_t^{i}(k,j)$ be the affine subspace generated by $\Lambda_{t-1}^{n-i-1}$ and points on $\Gamma_t^i$
whose $k$-th coordinates are $j$.
 We will show that each $\mathcal H_t^{i}(k,j)$ is indeed a hyperplane by calculating its equation.
 Suppose that one equation of $\mathcal H_t^{i}(k,j)$ is given by
\begin{equation}\label{eqn10}
a^{(t,i)}_1x_1+\cdots+a^{(t,i)}_nx_n=1.
\end{equation}

By definition, the vertices $\frac{(t-1)d}{s}\mathbf e_1,\ldots,\frac{(t-1)d}{s}\mathbf e_{i+1}$ of $\Lambda_{t-1}^{n-i-1}$ lie on $\mathcal H_t^{i}(k,j)$.
 We have
 $$a^{(t,i)}_1=\cdots=a^{(t,i)}_{i+1}=\frac{s}{(t-1)d}.$$
If $k=i+1$, then the following $(n-i-1)$ points lie on $\mathcal H_t^{i}(k,j)$:
 $$(0,\cdots,0,j,\frac{td}{s}-j,0,\cdots,0),\ldots,(0,\cdots,0,j,0,\cdots,0,\frac{td}{s}-j),$$
where $j$ is located at each $(i+1)$-th coordinate.
We have that
$$\frac{sj}{(t-1)d}+a^{(t,i)}_{i+2}(\frac{td}{s}-j)=\cdots= \frac{sj}{(t-1)d}+a^{(t,i)}_{n}(\frac{td}{s}-j)=1.$$
Hence
$$a^{(t,i)}_{i+2}=\cdots=a^{(t,i)}_{n}=\frac{s((t-1)d-sj)}{d(t-1)(td-sj)}.$$

Suppose $k>i+1$. For $\tilde{k}\in \{i+2,\cdots,n\}\setminus\{k\}$, inputting the coordinates of points whose $k$-th coordinate is $j$, $\tilde k$-th coordinate is $\frac{td}{s}-j$ and other coordinates are zeros, we solve that
$$a_j^{(t,i)}=\frac{s}{(t-1)d} \text{ for } j\neq k \text{ and } a_k=\frac{sj-d}{(t-1)dj}.$$

The above calculation shows that $\mathcal H_t^{i}(k,j)$ is only determined by one linear equation. In other words, $\mathcal H_t^{i}(k,j)$ is a hyperplane in $\mathbb{R}^n$.

For $1\le l\le i+1, 1\le m\le \frac{(t-1)d}{s}-1$, let
$\mathcal P_t^{i}(l,m)$ be the affine subspace generated by $\Gamma_t^i$  and points in $\Lambda_{t-1}^{n-i-1}$ whose $l$-th coordinate is $m$. One calculates that $\mathcal P_t^{i}(l,m)$ is only determined by  one equation
\begin{equation}\label{bbb}
b^{(t,i)}_1x_1+\cdots+b^{(t,i)}_nx_n=1,
\end{equation}
whose coefficients are given as follows.
If $l=i+1$, then $$b^{(t,i)}_{i+1}=\cdots=b^{(t,i)}_{n}=\frac{s}{td}\ {\rm and}\  b^{(t,i)}_{1}=\cdots=b^{(t,i)}_{i}=\frac{s(td-sm)}{td((t-1)d-sm)}>b^{(t,i)}_{l}.$$
If  $l<i+1$, then
$$b^{(t,i)}_j=\frac{s}{td} \text{ for } j\neq l\ {\rm and}\ b^{(t,i)}_l=\frac{s}{td}+\frac{1}{tm}>b^{(t,i)}_j.
$$

The series of hyperplanes $$\big\{v_k=j\big\}_{i+1\le k\le n; 1\le j\le \frac{td}{s}-1}$$
divide $\Gamma_t^i$ into $(\frac{td}{s})^{n-i-1}$ unit simplices of dimension $(n-i-1)$. Denote the set of these $(\frac{td}{s})^{n-i-1}$ unit simplices by $\Upsilon_t^i$. The series of hyperplanes $$\big\{v_l=m\big\}_{1\le l \le i+1; 1\le m\le \frac{(t-1)d}{s}-1}$$ divide $\Lambda_{t-1}^{n-i-1}$ into $(\frac{(t-1)d}{s})^{i}$ unit simplices of dimension $(n-i-1)$.
Denote the set of these $(\frac{(t-1)d}{s})^{i}$ unit simplices by $\Phi_{t-1}^{n-i-1}$. We have
\begin{lem}\label{alphabeta}
$$P(\Gamma_t^i,\Lambda_{t-1}^{n-i-1})=\bigcup_{\alpha\in \Upsilon_t^i,\beta\in \Phi_{t-1}^{n-i-1}}P(\alpha,\beta)
.$$
\end{lem}
\begin{proof}
Suppose that  $\mathbf{v}_{1},\cdots,\mathbf{v}_{n-i}$ are vertices of $\Gamma_{t}^i$, and $\mathbf{v}_{n-i+1},\cdots,\mathbf{v}_{n+1}$ are vertices of $\Lambda_{t-1}^{n-i-1}$. For any $\mathbf{x}\in P(\Gamma_t^i,\Lambda_{t-1}^{n-i-1})$, we have
$$\mathbf{x}=\sum_{j=1}^{n-i}\lambda_j\mathbf{v}_j+\sum_{j=n-i+1}^{n+1}\lambda_j\mathbf{v}_j$$
for some $\lambda_j\geq 0$ such that $\sum_{j=1}^{n+1}\lambda_j=1$.
Let $\lambda=\lambda_1+\cdots+\lambda_{n-i}$,
$\mathbf{x}_1=\frac{1}{\lambda}\sum_{j=1}^{n-i}\lambda_j\mathbf{v}_j$ and $\mathbf{x}_2=\frac{1}{1-\lambda}\sum_{j=n-i+1}^{n+1}\lambda_j\mathbf{v}_j.$ We have $\mathbf x_1\in \Gamma_t^i$, $\mathbf x_2\in \Lambda_{t-1}^{n-i-1}$ and $\mathbf x=\lambda \mathbf x_1+(1-\lambda)\mathbf x_2.$ Choose $\alpha\in \Upsilon_t^i$, $\beta\in \Phi_{t-1}^{n-i-1}$ such that $\mathbf x_1\in \alpha$, $\mathbf x_2\in \beta$, we get $\mathbf x\in P(\alpha, \beta)$. Hence
$P(\Gamma_t^i,\Lambda_{t-1}^{n-i-1})\subseteq \bigcup_{\alpha\in \Upsilon_t^i,\beta\in \Phi_{t-1}^{n-i-1}}P(\alpha,\beta)$.
The other direction is evident.
This proves Lemma \ref{alphabeta}.
\end{proof}
Note that each $\alpha\in \Upsilon_t^i$ (resp. $\beta \in \Phi_{t-1}^{n-i-1}$) lies on one side of each hyperplane $\{v_k=j\}$ (resp. $\{v_l=m\}$). Hence $P(\alpha, \beta)$ lies on one side of each hyperplane $\mathcal H_t^{i}(k,j)$ (resp.  $\mathcal P_t^{i}(l,m)$).
For different pairs $(\alpha, \beta)$ and $(\alpha', \beta')$,
if $\alpha\neq \alpha'$, then we choose $(k,j)$ such that $\alpha$ and $\alpha'$ lie on the opposite sides of hyperplane $\{v_k=j\}$ in $\Gamma_t^i$, which means that $P(\alpha,
\beta)$ and $P(\alpha',
\beta')$ lie on the opposite sides of $\mathcal{H}_t^{i}(k,j)$.
If $\beta\neq \beta'$, then we can similarly choose $\mathcal P_t^{i}(l,m)$ such that $P(\alpha, \beta)$ and $P(\alpha',
\beta')$ lie on the opposite sides of $\mathcal P_t^{i}(l,m)$.

Hence the hyperplanes $$\Big\{\mathcal H_t^{i}(k,j), \mathcal P_t^{i}(l,m)\Big\}_{\substack{i+1\le k\le n;1\le j\le td/s-1\\ 1\le l\le i+1; 1\le m\le (t-1)d/s-1}}$$ divide $P(\Gamma_t^i, \Lambda_{t-1}^{n-i-1})$ into a regular subdivision (see Figure 1. for the case $A=B=1$, $d=2$ and $n=3$):
\begin{equation*}
P(\Gamma_t^i,\Lambda_{t-1}^{n-i-1})=\bigcup_{\alpha\in \Upsilon_t^i,\beta\in \Phi_{t-1}^{n-i-1}}P(\alpha,\beta).
\end{equation*}
 \begin{figure}[t]
\label{Fig8}
\begin{center}
\begin{tikzpicture}
\draw [fill=black](1,0,0) circle (0.05cm);
\draw [fill=black](0,1,0) circle (0.05cm);
\draw [fill=black](0,0,1) circle (0.05cm);
 \node (n1) at (1.2,0,0) {$(1,0,0)$};
 \node (n1) at (0,1.4,0) {$(0,1,0)$};
 \node (n1) at (0,0,1.6) {$(0,0,1)$};
\draw [fill=black](4,0,0) circle (0.05cm);
\draw [fill=black](2,2,0) circle (0.05cm);
\draw [fill=black](2,0,2) circle (0.05cm);
\draw [fill=black](3,0,1) circle (0.05cm);
\draw [fill=black](3,1,0) circle (0.05cm);
\draw [fill=black](2,1,1) circle (0.05cm);
\node (n1) at (4.6,0,0) {$(2,0,0)$};
\node(n1) at (2,-0.5,2) {$(0,0,2)$};
\node(n1) at (3,2.5,1) {$(0,2,0)$};
\draw [thick, dashed](0,0,1)--(1,0,0);
\draw [black](0,0,1)--(0,1,0);
\draw [thick, dashed](1,0,0)--(0,1,0);
\draw [black](2,0,2)--(4,0,0);
\draw [black](2,0,2)--(2,2,0);
\draw [black](4,0,0)--(2,2,0);
\draw [black](0,1,0)--(2,0,2);
\draw [black](0,1,0)--(2,2,0);
\begin{scope}[fill opacity=.4, draw opacity=.5, text opacity=1]
 \draw [fill=yellow](0,1,0)--(2,0,2)--(4,0,0)-- cycle;
  \draw  [thick, dotted](0,1,0)--(3,1,0);
   \draw  [black](0,1,0)--(2,1,1);
   \draw  [black](2,1,1)--(3,1,0);
  \draw  [thick, dotted](0,1,0)--(3,0,1);
      \draw  [black](2,1,1)--(3,0,1);
    \draw  [black](3,1,0)--(3,0,1);
      \draw  [black](0,0,1)--(2,0,2);
       \draw  [thick, dashed](1,0,0)--(4,0,0);
          \draw  [thick, dotted](0,0,1)--(3,0,1);
  \end{scope}
  \node (n1) at (2.5,1.6,0.5) {$\Gamma$};
    \node (n2) at (0.2,0.4,0.3) {$\Lambda$};
      \draw [->, shorten >=5mm, shorten <=5mm] (2.2,0.5,1.5) [bend right] to (3,-1.2,1.5) node at (3.3,-1,1.5) { $\mathcal H_2^1$ };
        \draw [->, shorten >=5mm, shorten <=5mm] (0.3,0.3,0.3) [bend left] to (-1.3,0.3,1.5) node at (-1.3,0.3,1.5) { $\mathcal H_2^2$ };
  \begin{scope}[fill opacity=.2, draw opacity=.3, text opacity=1]
   \draw [fill=blue](0,1,0)--(0,0,1)--(4,0,0)-- cycle;
     \end{scope}
 \end{tikzpicture}
\end{center}
\caption{ A complete regular subdivision of $\tilde\delta_{d}(2)$ when $A=B=1$, $d=2$ and $n=3$}
 \end{figure}
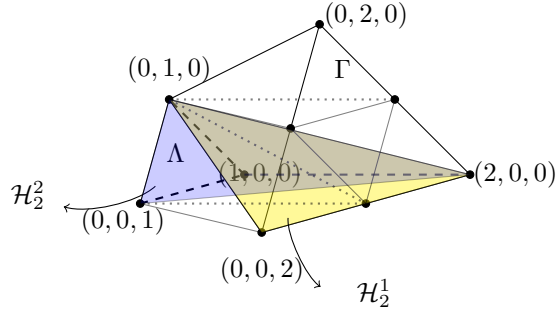

For $0\le i\le n-1, \alpha\in \Upsilon_t^i, \beta \in \Phi_{t-1}^{n-i-1},$ let
$$\hat\alpha:=\rho^{-1}(\alpha)=\Big\{\Big(\frac{t(A+B)}{s}-B,\mathbf v\Big):\mathbf v\in \alpha\Big\}$$ and $$\hat\beta:=\rho^{-1}(\beta)=\Big\{\Big(\frac{(t-1)(A+B)}{s}-B,\mathbf v\Big):
\mathbf v\in \beta \Big\}.$$
Let $P(\hat\alpha, \hat\beta)$ be the convex hull of $\hat\alpha$ and $\hat\beta$. Then $P(\hat\alpha, \hat\beta)=\rho^{-1}\big(P(\alpha, \beta)\big)$ is an integral polytope. Hence we get a regular subdivision (see Figure 2. for the case $A=2$, $B=1$, $d=6$ and $n=2$)
\begin{equation}\label{eqnn7}
P(\hat{\Gamma}_t^i,\hat{\Lambda}_{t-1}^{n-i-1})=\bigcup_{\alpha\in \Upsilon_t^i, \beta\in\Phi_{t-1}^{n-i-1}} P(\hat{\alpha},\hat{\beta}).
\end{equation}

\begin{figure}[t]
\label{Fig7}
\begin{center}
\begin{tikzpicture}[scale=0.8]
  \begin{scope}[fill opacity=.2, draw opacity=.3, text opacity=1]
 \draw [fill=yellow](0,0,-2)--(6,0,-2)--(0,6,-2)-- cycle;
 \end{scope}
\draw [fill=black](0,0,1) circle (0.05cm) node[left]{(-1,0,0)};
\draw [fill=black](0,0,0) circle (0.05cm);
\draw [fill=black](0,0,-1) circle (0.05cm);
\draw [fill=black](0,0,-2) circle (0.05cm) node[left]{(2,0,0)};
\draw [fill=black](2,0,0) circle (0.05cm)
node[below]{(0,0,2)};
\draw [fill=black](4,0,-1) circle (0.05cm)
node[below]{(1,0,4)};
\draw [fill=black](6,0,-2) circle (0.05cm)
node[right]{(2,0,6)};

\draw [fill=black](1,1,0) circle (0.05cm);
\draw [fill=black](0,2,0) circle (0.05cm)
node[left]{(0,2,0)};
\draw [fill=black](0,4,-1) circle (0.05cm)
node[left]{(1,4,0)};
\draw [fill=black](0,6,-2) circle (0.05cm)
node[left]{(2,6,0)};

\draw [fill=black](3,1,-1) circle (0.05cm);
\draw [fill=black](2,2,-1) circle (0.05cm);
\draw [fill=black](1,3,-1) circle (0.05cm);

\draw [fill=black](5,1,-2) circle (0.05cm);
\draw [fill=black](4,2,-2) circle (0.05cm);
\draw [fill=black](3,3,-2) circle (0.05cm);
\draw [fill=black](2,4,-2) circle (0.05cm);
\draw [fill=black](1,5,-2) circle (0.05cm);

\draw[thick,-](4,0,-1)--(0,4,-1);
\draw[thick,-](6,0,-2)--(0,6,-2);
\draw[thick,-](6,0,-2)--(0,0,1);
\draw[thick,-](0,6,-2)--(0,0,1);
\draw[ultra thin](0,0,-2)--(0,0,1);
\draw[thick,-](2,0,0)--(0,2,0);

\draw[thick,-](4,0,-1)--(0,2,0);
\draw[dashed](3,1,-1)--(0,2,0);
\draw[dashed](4,0,-1)--(1,1,0);
\draw[dashed](2,2,-1)--(0,2,0);
\draw[dashed](1,3,-1)--(0,2,0);

 \draw [->, shorten >=5mm, shorten <=5mm] (1,4,-1.5) [bend left] to (-1,5,-2) node at (-1,5,-2) { $\delta_d(3)$ };

 \draw [dashed,->, shorten >=5mm, shorten <=5mm] (0.5,3.5,-2.5) [bend left] to (2.5,5,-2) node at (2.5,5,-2) { $\delta'_d$ };

 \draw [->, shorten >=5mm, shorten <=5mm] (1,0,-1) [bend right] to (1,-1,1) node at (1,-1,1) { $\delta_d(1)$ };

 \draw [->, shorten >=5mm, shorten <=5mm] (2.5,0.5,-0.4) [bend left] to (4,2,-1) node at (4,2,-1) { $P(\hat\alpha, \hat\Lambda_1^1)$ };
  \begin{scope}[fill opacity=.2, draw opacity=.3, text opacity=1]
 \draw [fill=blue](0,2,0)--(4,0,-1)--(3,1,-1)-- cycle;
 \end{scope}
  \begin{scope}[fill opacity=.2, draw opacity=.3, text opacity=1]
 \draw [fill=blue](1,1,0)--(4,0,-1)--(2,0,0)-- cycle;
 \end{scope}

\draw [->, shorten >=5mm, shorten <=5mm] (2,0.5,-0.5) [bend left] to (4,0,2) node at (4,0,2) { $P(\hat\Gamma_2^1, \hat\beta)$ };
\draw [->, shorten >=5mm, shorten <=5mm] (1.5,2.5,-1) [bend left] to (-1,3,-1) node at (-1,3,-1) { $\delta_d(2)$ };
 \end{tikzpicture}
\end{center}
\caption{ A complete regular subdivision of $\delta_d(2)$ when $A=2, B=1$, $d=6$ and $n=2$ }
 \end{figure}

Combining (\ref{big}), (\ref{eqn6}) and (\ref{eqnn7}), we have a decomposition:
\begin{equation}\label{eqn8}
\delta_d=\bigcup_{t=1}^s\bigcup_{i=0}^{n-1}\bigcup_{\alpha\in \Upsilon_t^i, \beta\in\Phi_{t-1}^{n-i-1}} P(\hat{\alpha},\hat{\beta}).
\end{equation}
 We have proved that (\ref{big}) is a regular subdivision of $\delta_d$, (\ref{eqn6}) is a regular subdivision of $\delta_d(t)$ and (\ref{eqnn7})
 is a regular subdivision of $P(\hat{\Gamma}_t^i,\hat{\Lambda}_{t-1}^{n-i-1})$.
It follows from Lemma \ref{lem4} that to prove (\ref{eqn8}) is a regular subdivision, it suffices to prove
(\ref{eqn8}) is a subdivision, and this follows from Lemma \ref{lemma2}.
\end{proof}

Using the same method as the proof of Lemma \ref{alphabeta}, we can prove
\begin{lem}\label{lemma1}
For any $\mathbf{u}\in P(\hat\alpha, \hat\beta)$, there exist points  $\mathbf{a}\in \hat\alpha$ and $\mathbf{b}\in \hat\beta$ such that $\mathbf u=\lambda \mathbf{a}+(1-\lambda)\mathbf{b}$ for some real number $\lambda\in [0,1].$
\end{lem}
It follows from Lemma \ref{lemma1} that any lattice point in $P(\hat\alpha, \hat\beta)$ must lie in $\hat{\alpha}$ or $\hat{\beta}$ by considering its first coordinate. Since simplices $\alpha$ and $\beta$
contain no lattice points other than vertices, one has $P(\hat\alpha, \hat\beta)$ is an indecomposable simplex.

\begin{lem}\label{lemma2}
Choose two pieces $P(\hat\alpha_1, \hat\beta_1)$ and $P(\hat\alpha_2, \hat\beta_2)$ of (\ref{eqn8}) such that $P(\hat\alpha_1, \hat\beta_1)\subseteq P(\hat\Gamma_{t_1}^i, \hat\Lambda_{t_1-1}^{n-i-1})$ and $P(\hat\alpha_2, \hat\beta_2)\subseteq P(\hat\Gamma_{t_2}^j, \hat\Lambda_{t_2-1}^{n-j-1})$.
If $t_1=t_2$, then
$$P(\hat\alpha_1, \hat\beta_1)\cap P(\hat\alpha_2, \hat\beta_2)=P(\hat\alpha_{12}, \hat\beta_{12}),$$
where $\alpha_{12}=\alpha_1\cap\alpha_2, \beta_{12}=\beta_1\cap\beta_2.$
If $|t_1-t_2|=1$, then
$ P(\hat\alpha_1, \hat\beta_1)\cap P(\hat\alpha_2, \hat\beta_2)=\hat{\beta}_1\cap \hat{\alpha}_2$ or $\hat{\beta}_2\cap \hat{\alpha}_1$.
If $|t_1-t_2|>1$,  then $ P(\hat\alpha_1, \hat\beta_1)\cap P(\hat\alpha_2, \hat\beta_2)=\emptyset.$ In particular, $P(\hat\alpha_1, \hat\beta_1)\cap P(\hat\alpha_2, \hat\beta_2)$ is a common face of $P(\hat\alpha_1, \hat\beta_1)$ and $P(\hat\alpha_2, \hat\beta_2)$.
\end{lem}
\begin{proof}
We only prove the case $t_1=t_2$, the others are similar and much simpler.
It is obvious that $P(\hat\alpha_{12}, \hat\beta_{12})\subseteq P(\hat\alpha_1, \hat\beta_1)\cap P(\hat\alpha_2, \hat\beta_2)$.
It remains to show the other direction.
Let $\mathbf{u}\in P(\hat\alpha_1, \hat\beta_1)\cap P(\hat\alpha_2, \hat\beta_2).$
By Lemma \ref{lemma1}, there exist $\mathbf{a}_1\in \hat{\alpha}_1$, $\mathbf{b}_1\in \hat{\beta}_1$, $\mathbf{a}_2\in \hat{\alpha}_2$,
$\mathbf{b}_2\in \hat{\beta}_2$, and $0\le \lambda,\mu\le 1$ satisfying that
\begin{equation}\label{eqnu}
 \mathbf{u}=\lambda \mathbf{a}_1+(1-\lambda) \mathbf{b}_1=\mu \mathbf{a}_2+(1-\mu) \mathbf{b}_2.
\end{equation}

Comparing the first coordinate of elements of (\ref{eqnu}), we have $\lambda=\mu$. Hence
\begin{equation*}
\lambda (\mathbf{a}_1-\mathbf{a}_2)=(1-\lambda) (\mathbf{b}_2-\mathbf{b}_1).
\end{equation*}
We divide the proof into two cases.

{\sc Case 1.} $j=i$.
The first $i$ coordinates of the left hand side and the last $(n-i-1)$ coordinates of the right hand side is $0$. Thus
$\mathbf{a}_1=\mathbf{a}_2\in \hat\alpha_1\cap\hat\alpha_2=\hat{\alpha}_{12}$ and
$\mathbf{b}_1=\mathbf{b}_2\in\hat\beta_1\cap\hat\beta_2= \hat{\beta}_{12}$. Then $\mathbf u\in P(\hat\alpha_{12}, \hat\beta_{12})$. Thus $P(\hat\alpha_1, \hat\beta_1)\cap P(\hat\alpha_2, \hat\beta_2)\subseteq P(\hat\alpha_{12}, \hat\beta_{12})$.

{\sc Case 2.} $j\neq i$. Without loss of generality, we assume $j>i$.
Note that the hyperplane $\hat{\mathcal{H}}^j_{t_1}$ passing through $\hat{\Gamma}_{t_1}^j$ and $\hat{\Lambda}_{t_1-1}^{n-j}$ separates $P(\hat{\Gamma}_{t_1}^i,\hat{\Lambda}_{t_1-1}^{n-i-1})$ and
$P(\hat{\Gamma}_{t_1}^j,\hat{\Lambda}_{t_1-1}^{n-j-1})$.
Then $\mathbf{b}_1\in \hat{\Lambda}_{t_1-1}^{n-i-1}\subseteq \hat{\Lambda}_{t_1-1}^{n-j}\subseteq \hat{\mathcal{H}}^j_{t_1}$.
Note that $\mathbf{u}\in \hat{\mathcal{H}}^j_{t_1}$.
Then $\mathbf{a}_1\in \hat{\mathcal{H}}^j_{t_1}$. It follows from Lemma \ref{hyperplane} that $\mathbf{a}_1\in \hat{\Gamma}_{t_1}^j$ .
On the other hand, $\hat{\mathcal{H}}^{i+1}_{t_1}$ passing through $\hat{\Gamma}_{t_1}^{i+1}$ and $\hat{\Lambda}_{t_1-1}^{n-i-1}$ also separates $P(\hat{\Gamma}_{t_1}^i,\hat{\Lambda}_{t_1-1}^{n-i-1})$ and
$P(\hat{\Gamma}_{t_1}^j,\hat{\Lambda}_{t_1-1}^{n-j-1})$. Note that $\mathbf{u}\in \hat{\mathcal{H}}^{i+1}_{t_1}$.
Since $\mathbf{a}_2\in \hat{\Gamma}_{t_1}^j\subseteq \hat{\Gamma}_{t_1}^{i+1}\subseteq \hat{\mathcal{H}}^{i+1}_{t_1}$, one has $\mathbf{b}_2\in \hat{\Lambda}_{t_1-1}^{n-i-1}$.
We conclude that $\mathbf{a}_1,\mathbf{a}_2\in \hat{\Gamma}_{t_1}^j$ and $\mathbf{b}_1,\mathbf{b}_2\in \hat{\Lambda}_{t_1-1}^{n-i-1}$. Then using the same argument as Case 1, we conclude that $\mathbf{u}\in P(\hat\alpha_{12}, \hat\beta_{12})$.
Then $P(\hat\alpha_1, \hat\beta_1)\cap P(\hat\alpha_2, \hat\beta_2)\subseteq P(\hat\alpha_{12}, \hat\beta_{12})$.
\end{proof}

\begin{prop}\label{lem1}
Let $n\ge 2$. The Adolphson--Sperber conjecture is true for $\Delta_d$.
\end{prop}
\begin{proof}
Let $\alpha\in \Upsilon_t^i, \beta\in\Phi_{t-1}^{n-i-1}$, and $P^0(\hat{\alpha},\hat{\beta})$ be the convex closure in $\mathbb{R}^{n+1}$ generated by the origin and lattice points in $P(\hat{\alpha},\hat{\beta})$.
Denote the vertices of $\hat \beta$ by
$$\mathbf{v}_j=\Big(\frac{(A+B)(t-1)-sB}{s},v_{j,1},\cdots,v_{j,i+1},0,\cdots,0\Big)$$
for $j=1,\ldots,i+1$, the vertices of $\hat\alpha$ by
$$\mathbf{v}_j=\Big(\frac{(A+B)t-sB}{s},0,\cdots,0,v_{j,i+1},\cdots,v_{j,n}\Big)$$
for $j=i+2,\ldots,n+1.$ Using matrix operations, we have
\begin{align*}
(n+1)!{\rm Vol}(P^0(\hat{\alpha},\hat{\beta}))
=&\pm\begin{vmatrix}
\mathbf{v}_1, \cdots, \mathbf{v}_{i+1},  \mathbf{v}_{i+2},   \cdots,  \mathbf{v}_{n+1}
\end{vmatrix}\\
=&\pm\frac{Bs}{t(t-1)d}\begin{vmatrix}
v_{1,1} & \cdots & v_{i+1,1} \\
\vdots&\cdots&\vdots\\
v_{1,i+1} & \cdots & v_{i+1,i+1}
\end{vmatrix}\cdot \begin{vmatrix}
v_{i+2,i+1} &  \cdots & v_{n+1,i+1}\\
\vdots&\vdots&\vdots\\
v_{i+2,n}  &\cdots& v_{n+1,n}
\end{vmatrix},
\end{align*}
where each $\mathbf{v}_j$ is written as a column vector.
Denote by $p_{i+1}: \mathbb R^n\to \mathbb R^{i+1}$ be the projection of $\mathbb R^n$ onto its first $i+1$ coordinates. Let $ \rho_{i+1}=p_{i+1}\circ \rho:L_d\to \mathbb{R}^{i+1}$ be the composition of $\rho$ and $p_{i+1}$.
Then $\rho_{i+1}(\hat{\Lambda}_{t-1}^{i})=
p_{i+1}(\Lambda_{t-1}^{i})\subseteq \mathbb R^{i+1}$
is the simplex with vertices
$$\Big(\frac{(t-1)d}{s},0,\cdots,0\Big), \Big(0,\frac{(t-1)d}{s},0,\cdots,0\Big),\ldots, \Big(0,0,\cdots,0,\frac{(t-1)d}{s}\Big).$$
Then $\rho_{i+1}(\hat\beta)=p_{i+1}(\beta)\in \mathbb R^{i+1}$ is the simplex with vertices
$$(v_{1,1},v_{1,2},\cdots,v_{1,i+1}),(v_{2,1},v_{2,2},\cdots,v_{2,i+1}),\ldots,(v_{i+1,1},v_{i+1,2},\cdots,v_{i+1,i+1}),$$
which is a cell of the hyperplane decompositions of $\rho_{i+1}(\hat\Lambda^i)$ cut by $\{v_l=m\}$ for $l=1,\ldots,i+1$ and $m=1,\ldots,\frac{(t-1)d}{s}-1$.
Calculating the volume of $\rho_{i+1}(\hat\beta)$, we have
$$\begin{vmatrix}
v_{1,1} & \cdots & v_{i+1,1}\\
\vdots&\cdots&\vdots\\
v_{1,i+1} & \cdots & v_{i+1,i+1}
\end{vmatrix}=\pm(i+1)!{\rm Vol}(\rho_{i+1}(\hat\beta))=\pm\frac{(t-1)d}{s}.$$
Similarly, we can prove
\begin{align*}
\begin{vmatrix}
v_{i+2,i+1} &  \cdots & v_{n+1,i+1}\\
\vdots&\cdots&\vdots\\
 v_{i+2,n}  &\cdots& v_{n+1,n}
\end{vmatrix}=\pm \frac{td}{s}.
\end{align*}
Hence
$$(n+1)!{\rm Vol}(P^0(\hat{\alpha},\hat{\beta}))=\frac{dB}{s}.$$

It follows from Proposition \ref{prop2}, Corollary \ref{thm21} that if $p\equiv 1 \mod \frac{dB}{s}$, then $\Delta_d$ is generically ordinary.
This finishes the proof of Proposition \ref{lem1}.
\end{proof}

\begin{proof}[ Proof of Theorem \ref{thm1}.]

If $s=(A+B,d)=1$, it follows from \cite[Theorem 5.1]{FW} that Theorem \ref{thm1} holds.
If $s\ge 2$, it then follows from Proposition \ref{prop2} and \ref{lem1} that $\Delta_d$ is generically ordinary if $p\equiv 1 \mod \frac{dB}{s}$.
 It follows from \cite[Theorem 7.5]{WD1} that the Adolphson--Sperber conjecture is true for $\Delta_d'$.
The face $\delta_d'$ lies on the hyperplane $x_0=A$. Then $D(\Delta_d')=A$ and ${\rm GNP}(\Delta_d',p)={\rm HP}(\Delta_d')$  if $p\equiv 1 \mod A$.
 Note that $\Delta=\Delta_d\cup \Delta_d'$.
Then Theorem \ref{thm1} follows from the facial decomposition theorem.
This finishes the proof of Theorem \ref{thm1}.
\end{proof}

\begin{center}
{\sc Acknowledgements}
\end{center}

The authors would like to thank Professor Daqing Wan for many helpful  suggestions, Professor Jiyou Li for many
helpful discussions and Huaiqian Li for many helpful comments.  L.P. Yang is supported by Beijing Postdoctoral Research Foundation (No. 202022073).

\bibliographystyle{amsplain}

\end{document}